\documentclass{article}
\usepackage{amsfonts}
\usepackage{amsmath}
\usepackage{amsthm}
\usepackage{hyperref}
\usepackage{amssymb,amsthm}
\usepackage[mathscr]{euscript}
\usepackage{amssymb,amsmath,latexsym}
\usepackage[varg]{pxfonts}
\usepackage[bottom]{footmisc}
\usepackage[affil-it]{authblk}
\usepackage{mathtools}
\usepackage{cite}

\newtheorem{theorem}{Theorem}[section]

\newtheorem{definition}[theorem]{Definition}
\newtheorem{example}[theorem]{Example}

\newtheorem{lemma}[theorem]{Lemma}

\newtheorem{proposition}[theorem]{Proposition}

\numberwithin{equation}{section}
\begin{document}

\title{Solvability of Langevin equations with two Hadamard fractional derivatives via Mittag-Leffler functions}
\author{Mohamed I. Abbas\thanks{miabbas77@gmail.com} \\ Department of Mathematics and Computer Science,\\ 
Faculty of Science, Alexandria University, Alexandria 21511, Egypt\and Maria Alessandra Ragusa\thanks{mariaalessandra.ragusa@unict.it
}\\ Dipartimento di Matematica e Informatica, Università di Catania\\
      Viale Andrea Doria, 6, CATANIA - 95125, ITALY\\RUDN University, 6 Miklukho - Maklay St, Moscow, 117198, Russia}
\date{}

\maketitle
\makeatletter
\renewcommand\@makefnmark%
{\mbox{\textsuperscript{\normalfont\@thefnmark)}}}
\makeatother

\begin{abstract}
In this paper we discuss the solvability of Langevin equations with two Hadamard fractional derivatives. The method of  this discussion is to study the solutions of the equivalent Volterra integral equation in terms of Mittag-Leffler functions. The existence and uniqueness results are established by using  Schauder fixed point theorem and  Banach fixed point theorem respectively. An example is given to illustrate the main results.
\end{abstract}

\section{Introduction}
In recent years, The fractional differential equations have elicited as a rich area of research due to their applications in various fields of science and engineering such as aerodynamics, viscoelasticity, control theory, economics and blood flow phenomena, etc. For more details, we refer the reader to the text books \cite{Kilbas,Lakshmikantham,Miller,Podlubny}.\\

Besides the well-known Riemann-Liouville and Caputo type fractional derivatives, there is an other fractional derivative introduced by Hadamard in 1892 \cite{Hadamard}, which is known as Hadamard derivative and differs from the preceding ones in the sense that the kernel of the integral in its definition contains logarithmic function of arbitrary exponent. A detailed description of Hadamard fractional derivative and integral can be found in \cite{Kilbas,Butzer1,Butzer2,Butzer3} and the references therein. For some recent results in Hadamard fractional differential equations,  we refer the reader to \cite{Abbas1,Song,Zhai} and the references therein.\\

The Langevin equation was able to attract the attention of researchers due to its importance in mathematical physics that is used in modeling the phenomena occurring in fluctuating environment such as Brownian motion \cite{Coffey}.\\

 In 1908, P. Langevin \cite{Langevin} derived the classical form of this equation in terms of ordinary derivatives. For the systems in complex media, classical Langevin equation does not provide the correct description of the dynamics. The generalized Langevin equation was introduced in 1966 by Kubo \cite{Kubo}, where a frictional memory kernel was incorporated into the Langevin equation to describe the fractal and memory properties. In the 1990s, Mainardi and collaborators \cite{Mainardi1,Mainardi2} introduced the fractional Langevin equation. Many interesting results regarding the existence, uniqueness and stability results  for fractional order Langevin equations have been studied by many researchers, see for example \cite{Bashir1,Bashir2,Berhail,A,B,Liu,C,D} and references cited therein.\\

As far as we know, there are no contributions associated with the solutions of the equivalent Volterra integral equations of the fractional  Langevin equations in terms of Mittag-leffler functions.\\

The objective of this paper is to investigate the existence and uniqueness of solutions for the following  Langevin equations with two Hadamard fractional derivatives:
\begin{equation}\label{main}
\begin{cases}
\prescript{H}{}D_{1,t}^{\beta}(\prescript{H}{}D_{1,t}^{\alpha}-\lambda)x(t)=f(t,x(t)),~~t\in[1,e],~\lambda>0,~0<\alpha,\beta\leq 1,\\
(\prescript{H}{}D_{1,t}^{\alpha}-\lambda) x(e)=0,~~~\prescript{H}{}I_{1^+}^{1-\alpha}x(1^+)=c_{0},~~c_{0}\in\mathbb{R},
\end{cases}
\end{equation}
where  $\prescript{H}{}D_{1,t}^{\alpha}~,\prescript{H}{}D_{1,t}^{\beta}$ denote Hadamard fractional derivatives of orders $\alpha,~\beta~(0<\alpha,\beta\leq 1)$ respectively , $\prescript{H}{}I_{1^+}^{1-\alpha}$ denotes the left-sided Hadamard fractional integral of order $1-\alpha$ and $f:[1,e]\times\mathbb{R}\to \mathbb{R}$ is given continuous function.

\section{Preliminaries}
Let $C([1,e],\mathbb{R})$ be the Banach space of all continuous functions from $[1,e]$ to $\mathbb{R}$ endowed with the norm
$\|x\|_{C}=\sup_{t\in[1,e]}|x(t)|$. \\
For $0\leq\gamma<1$, we define the weighted space of functions $g$ on $[1,e]$ by
$$C_{\gamma,\ln}([1,e],\mathbb{R})=\{g:[1,e]\to \mathbb{R}|~(\ln t)^{\gamma}g(t)\in C([1,e],\mathbb{R})\}.$$
Clearly, $C_{\gamma,\ln}([1,e],\mathbb{R})$ is the Banach space with the norm
$$\|g\|_{\gamma,\ln}=\|(\ln t)^{\gamma}g(t)\|_{C}=\sup_{t\in [1,e]}|(\ln t)^{\gamma}g(t)|.$$
The following definitions are devoted to the basic concepts of Hadamard fractional integrals and fractional derivatives.
\begin{definition}(see \textbf{\cite{Kilbas}}, p.110)
The left-sided Hadamard fractional integral of order $\alpha\in\mathbb{R}^{+}$ of function $g(t)$ is defined by
\[
\prescript{H}{}I_{1^+}^{\alpha}g(t)=\frac{1}{\Gamma(\alpha)}\int_{1}^{t}\left(\ln \frac{t}{s}\right)^{\alpha-1}g(s)\frac{ds}{s},~~0<1<t\leq b<\infty,
\]
where $\Gamma(\cdot)$ is the Gamma function.
\end{definition}
\begin{definition}(see \textbf{\cite{Kilbas}}, p.111)
The left-sided Hadamard fractional derivative of order $\alpha\in[n-1, n),~ n \in \mathbb{Z}^{+}$ of function $g(t)$ is defined by
\[
\prescript{H}{}D_{1^+}^{\alpha}g(t)=\frac{1}{\Gamma(n-\alpha)}\left(t\frac{d}{dt}\right)^{n}\int_{1}^{t}\left(\ln \frac{t}{s}\right)^{n-\alpha+1}g(s)\frac{ds}{s},~~0<1<t\leq b<\infty,
\]
where $\Gamma(\cdot)$ is the Gamma function.
\end{definition}
\begin{proposition}(see \textbf{\cite{Kilbas}}, p.114,115)
Let $\alpha>0,~\beta>0$ and $0<1<b<\infty$. Then for $f\in C_{\gamma,\ln}([1,e],\mathbb{R})$, the following properties hold:
$$\prescript{H}{}D_{1,t}^{\alpha}\prescript{H}{}I_{1^+}^{\beta}f=\prescript{H}{}I_{1^+}^{\beta-\alpha}f ~~~\textit{and}~~~ \prescript{H}{}D_{1,t}^{\alpha}\prescript{H}{}I_{1^+}^{\alpha}f=f.$$
\end{proposition}

\begin{lemma}(see \textbf{\cite{Kilbas}}, Theorem 2.3, p.116)\label{lem:prop}
Let $\beta>0$ ,~$0<1<b<\infty$ and $f\in C_{\gamma,\ln}([1,e],\mathbb{R})$. Then

$$\prescript{H}{}I_{1^+}^{\beta}\prescript{H}{}D_{1,t}^{\beta}u(t)=u(t)-\sum_{j=1}^{n}c_{j}(\ln t)^{\beta-j},$$
where $c_{j}\in \mathbb{R}$ and $n-1<\beta<n$.
\end{lemma}
\subsection{Auxiliary Lemma}
\begin{lemma}\label{lem:aux}
Given $\sigma\in C_{\gamma,\ln}([1,e],\mathbb{R})$. Then the linear problem
\begin{equation}\label{linear}
\begin{cases}
\prescript{H}{}D_{1,t}^{\beta}(\prescript{H}{}D_{1,t}^{\alpha}-\lambda)x(t)=\sigma(t),~~1<t<e,\\
(\prescript{H}{}D_{1,t}^{\alpha}-\lambda) x(e)=0,~~~\prescript{H}{}I_{1^+}^{1-\alpha}x(1^+)=c_{0},~~c_{0}\in\mathbb{R},
\end{cases}
\end{equation}
is equivalent to the integral equation
\begin{equation}\label{solnE}
\begin{split}
x(t)&=c_{0}(\ln t)^{\alpha-1}\mathbb{E}_{\alpha,\alpha}(\lambda(\ln t)^{\alpha})
+\int_{1}^{t}\left(\ln \frac{t}{s}\right)^{\alpha-1}\mathbb{E}_{\alpha,\alpha}\left(\lambda\left(\ln \frac{t}{s}\right)^{\alpha}\right)\\
&\times \left[\frac{1}{\Gamma(\beta)}\int_{1}^{s}\left(\ln \frac{s}{\tau}\right)^{\beta-1}\sigma(\tau)\frac{d\tau}{\tau}-\frac{(\ln s)^{\beta-1}}{\Gamma(\beta)}\int_{1}^{e}\left(1-\ln\tau\right)^{\beta-1}\sigma(\tau)\frac{d\tau}{\tau}\right]\frac{ds}{s}
\end{split}
\end{equation}
\end{lemma}
\begin{proof}
Firstly, we apply the Hadamard fractional integral of order $\beta$ to both sides of (\ref{linear}) and using the result of Lemma \ref{lem:prop}, we get
\begin{equation}\label{c1}
(\prescript{H}{}D_{1,t}^{\alpha}-\lambda)x(t)=\prescript{H}{}I_{1^+}^{\beta}\sigma(t)+c_{1}(\ln t)^{\beta-1},
\end{equation}
where $c_{1}\in \mathbb{R}$. Using the boundary condition $(\prescript{H}{}D_{1,t}^{\alpha}-\lambda) x(e)=0$, we get
$$c_{1}=-\frac{1}{\Gamma(\beta)}\int_{1}^{e}\left(1-\ln s\right)^{\beta-1}\sigma(s)\frac{ds}{s}.$$
Then we can rewrite equation (\ref{c1}) in the form
\begin{equation}\label{Esolution}
(\prescript{H}{}D_{1,t}^{\alpha}-\lambda)x(t)=\frac{1}{\Gamma(\beta)}\int_{1}^{t}\left(\ln \frac{t}{s}\right)^{\beta-1}\sigma(s)\frac{ds}{s}-\frac{(\ln t)^{\beta-1}}{\Gamma(\beta)}\int_{1}^{e}\left(1-\ln s\right)^{\beta-1}\sigma(s)\frac{ds}{s}.
\end{equation}
Now, by (\cite{Kilbas}, p. 234, (4.1.89)–(4.1.95)), we conclude that equation (\ref{Esolution}) with the boundary condition $\prescript{H}{}I_{1^+}^{1-\alpha}x(1^+)=c_{0}$ has a solution $x\in C_{\gamma,\ln}([1,e],\mathbb{R})$ that satisfies the volterra integral equation (\ref{solnE}), which yields the required solution. By direct computation we can prove the converse. This ends the proof.
\end{proof}

\begin{lemma}(\cite{JinRong})\label{Ronglemma}
Let $\alpha,\gamma\in(0,1),~1<t_1<t_2\leq e$ and $\alpha-1+\frac{1}{q}>0,~\gamma<\frac{1}{p}$. Then
$$\int_{1}^{t_2}(\ln s)^{-\gamma}\bigg(\left(\ln \frac{t_2}{s}\right)^{\alpha-1}-\left(\ln \frac{t_1}{s}\right)^{\alpha-1}\bigg)\frac{ds}{s}\leq\left(\frac{\ln (t_2)^{1-p\gamma}}{1-p\gamma}\right)^{\frac{1}{p}}\left[\frac{(\ln t_2)^{q(\alpha-1)+1}-(\ln t_1)^{q(\alpha-1)+1}+\left(\ln \frac{t_1}{t_2}\right)^{q(\alpha-1)+1}}{q(\alpha-1)+1}\right],$$
where $\frac{1}{p}+\frac{1}{q}=1,~p,q>1.$
\end{lemma}

\begin{lemma}(see \cite{Gorenflo}, 4.2.3 and 4.2.7)\label{Gorenfloetal}
The Mittag-Leffler function $\mathbb{E}_{\alpha,\beta}$ satisfies the following identities:
\begin{itemize}
  \item [](i)~$\mathbb{E}_{\alpha,\beta}(z)=\frac{1}{\Gamma(\beta)}+z\mathbb{E}_{\alpha,\beta+\alpha}(z)$,
  \item [](ii)~$\mathbb{E}_{\frac{1}{2},1}(z)=\sum_{k=0}^{\infty}\frac{z^k}{\Gamma(\frac{k}{2}+1)}=e^{z^2}~\textnormal{erfc}(-z)$,
\end{itemize}
where \textnormal{erfc} is complementary to the error function \textnormal{erf}:
$$\textnormal{erfc}(z)=\frac{2}{\sqrt{\pi}}\int_{z}^{\infty}e^{-u^2}~du=1-\textnormal{erf}(z),~~z\in\mathbb{C}.$$
\end{lemma}
\begin{lemma}(\textbf{The mean value theorem for integrals})\label{MVTh}
If $f$ is a continuous function on the closed, bounded interval $[ a , b ]$, then there is at least one number $\xi $ in $( a , b  )$ for which
$$f(\xi)=\frac{1}{b-a}\int_{a}^{b}f(t)dt.$$
\end{lemma}

In the following Lemma, we present some useful integrals which are used further in this paper.
\begin{lemma}\label{lem:integrals}
Using the famous beta function $\mathbb{B}(m,n)=\int_{0}^{1}(1-z)^{m-1}z^{n-1}~dz$, we get
\begin{eqnarray*}
(i)&&\int_{1}^{t}\left(\ln \frac{t}{s}\right)^{\beta-1}(\ln s)^{-\gamma}\frac{ds}{s}=(\ln t)^{\beta-\gamma}\mathbb{B}(\beta,1-\gamma).\\
(ii)&&\int_{1}^{t}\left(\ln \frac{t}{s}\right)^{\beta-1}\frac{ds}{s}=\frac{1}{\beta}(\ln t)^{\beta}.\\
(iii)&&\int_{1}^{e}\left(1-\ln s\right)^{\beta-1}(\ln s)^{-\gamma}\frac{ds}{s}=\mathbb{B}(\beta,1-\gamma).\\
(iv)&&\int_{1}^{e}\left(1-\ln s\right)^{\beta-1}\frac{ds}{s}=\frac{1}{\beta}.
\end{eqnarray*}
\end{lemma}
To end this section, we present the Schauder fixed point theorem which plays a main tool in the existence of the solution of (\ref{main}).
\begin{lemma}(\textbf{The Schauder fixed point theorem~\cite{Granas}})\label{fixed}
Let $E$ be a Banach space and $Q$ be a nonempty bounded convex and closed subset of $E$, and $N:Q\to Q$ is a compact and continuous map. Then $N$ has at least one fixed point in $Q$.
\end{lemma}

\section{Main Results}
Let $B_{r}=\left\{x\in C_{\gamma,\ln}([1,e],\mathbb{R}):\|x\|_{\gamma,\ln}\leq r\right\}$ with
$r\geq\frac{\omega_1}{1-L_1\omega_2},$
where $$\omega_1=\left(|c_0|+\frac{L_2\mathbb{B}(\alpha,1+\beta)}{\Gamma(\beta+1)}+\frac{L_2\mathbb{B}(\alpha,\beta)}{\Gamma(\beta+1)}\right)\mathbb{E}_{\alpha,\alpha}(\lambda),$$
and
$$\omega_2=\frac{\mathbb{B}(\beta,1-\gamma)}{\Gamma(\beta)}\big(\mathbb{B}(\alpha,1+\beta-\gamma)+\mathbb{B}(\alpha,\beta)\big)\mathbb{E}_{\alpha,\alpha}(\lambda).$$

We define the operator $(\mathcal{H}x)(t):B_{r}\to C_{\gamma,\ln}([1,e],\mathbb{R})$ by

\begin{equation}\label{operator}
\begin{split}
(\mathcal{H}x)(t)&=c_{0}(\ln t)^{\alpha-1}\mathbb{E}_{\alpha,\alpha}(\lambda(\ln t)^{\alpha})
+\int_{1}^{t}\left(\ln \frac{t}{s}\right)^{\alpha-1}\mathbb{E}_{\alpha,\alpha}\left(\lambda\left(\ln \frac{t}{s}\right)^{\alpha}\right)\\
&\times \left[\frac{1}{\Gamma(\beta)}\int_{1}^{s}\left(\ln \frac{s}{\tau}\right)^{\beta-1}f(\tau,x(\tau))\frac{d\tau}{\tau}-\frac{(\ln s)^{\beta-1}}{\Gamma(\beta)}\int_{1}^{e}\left(1-\ln\tau\right)^{\beta-1}f(\tau,x(\tau))\frac{d\tau}{\tau}\right]\frac{ds}{s}
\end{split}
\end{equation}
\begin{theorem}\label{th:existence}
Assume that:
\begin{itemize}
  \item []$\textbf{(H1)}$ ~$f:[1,e]\times\mathbb{R}\to \mathbb{R}$ is a continuous function, $f(\cdot,x(\cdot))\in C_{\gamma,\ln}([1,e],\mathbb{R})$ and there exist constants $0\leq L_1<\omega_{2}^{-1},~L_2>0$ such that
  $$|f(t,x)|\leq L_1|x|+L_2,~~\forall (t,x)\in[1,e]\times\mathbb{R}.$$
Then (\ref{main}) has at least one solution on $[1,e]$.
\end{itemize}
\end{theorem}
\begin{proof}
The proof will be through several steps.\\

\textbf{Step 1.} We show that $\mathcal{H}(B_{r})\subset B_{r}$.\\

For any $x\in C_{\gamma,\ln}([1,e],\mathbb{R})$ and Lemma \ref{lem:integrals}, we have\\

$\left|(\ln t)^{\gamma}(\mathcal{H}x)(t)\right|$
\begin{eqnarray*}
&\leq&\left|c_{0}(\ln t)^{\gamma+\alpha-1}\mathbb{E}_{\alpha,\alpha}(\lambda(\ln t)^{\alpha})\right|
+\left|(\ln t)^{\gamma}\int_{1}^{t}\left(\ln \frac{t}{s}\right)^{\alpha-1}\mathbb{E}_{\alpha,\alpha}\left(\lambda\left(\ln \frac{t}{s}\right)^{\alpha}\right)\left[\frac{1}{\Gamma(\beta)}\int_{1}^{s}\left(\ln \frac{s}{\tau}\right)^{\beta-1}f(\tau,x(\tau))\frac{d\tau}{\tau}\right.\right.\\
&-&\left.\left.\frac{(\ln s)^{\beta-1}}{\Gamma(\beta)}\int_{1}^{e}\left(1-\ln\tau\right)^{\beta-1}f(\tau,x(\tau))\frac{d\tau}{\tau}\right]\frac{ds}{s}\right|\\
&\leq& |c_0|(\ln t)^{\gamma+\alpha-1}\mathbb{E}_{\alpha,\alpha}(\lambda)+(\ln t)^{\gamma}\int_{1}^{t}\left(\ln \frac{t}{s}\right)^{\alpha-1}\mathbb{E}_{\alpha,\alpha}\left(\lambda\left(\ln \frac{t}{s}\right)^{\alpha}\right)\left[\frac{1}{\Gamma(\beta)}\int_{1}^{s}\left(\ln \frac{s}{\tau}\right)^{\beta-1}|f(\tau,x(\tau))|\frac{d\tau}{\tau}\right.\\
&+&\left.\frac{(\ln s)^{\beta-1}}{\Gamma(\beta)}\int_{1}^{e}\left(1-\ln\tau\right)^{\beta-1}|f(\tau,x(\tau))|\frac{d\tau}{\tau}\right]\frac{ds}{s}\\
&\leq& |c_0|\mathbb{E}_{\alpha,\alpha}(\lambda)+\int_{1}^{t}\left(\ln \frac{t}{s}\right)^{\alpha-1}\mathbb{E}_{\alpha,\alpha}\left(\lambda\left(\ln \frac{t}{s}\right)^{\alpha}\right)\left[\frac{1}{\Gamma(\beta)}\int_{1}^{s}\left(\ln \frac{s}{\tau}\right)^{\beta-1}(L_1|x(\tau)|+L_2)\frac{d\tau}{\tau}\right.\\
&+&\left.\frac{(\ln s)^{\beta-1}}{\Gamma(\beta)}\int_{1}^{e}\left(1-\ln\tau\right)^{\beta-1}(L_1|x(\tau)|+L_2)\frac{d\tau}{\tau}\right]\frac{ds}{s}
\end{eqnarray*}
\begin{eqnarray*}
&\leq& |c_0|\mathbb{E}_{\alpha,\alpha}(\lambda)+\int_{1}^{t}\left(\ln \frac{t}{s}\right)^{\alpha-1}\mathbb{E}_{\alpha,\alpha}\left(\lambda\left(\ln \frac{t}{s}\right)^{\alpha}\right)\\
&\times& \left[\frac{L_1}{\Gamma(\beta)}\int_{1}^{s}\left(\ln \frac{s}{\tau}\right)^{\beta-1}(\ln \tau)^{-\gamma}\|x\|_{\gamma,\ln}\frac{d\tau}{\tau}+\frac{L_2}{\Gamma(\beta)}\int_{1}^{s}\left(\ln \frac{s}{\tau}\right)^{\beta-1}\frac{d\tau}{\tau}\right.\\
&+&\left.\frac{L_1(\ln s)^{\beta-1}}{\Gamma(\beta)}\int_{1}^{e}\left(1-\ln\tau\right)^{\beta-1}(\ln \tau)^{-\gamma}\|x\|_{\gamma,\ln}\frac{d\tau}{\tau}+\frac{L_2(\ln s)^{\beta-1}}{\Gamma(\beta)}\int_{1}^{e}\left(1-\ln\tau\right)^{\beta-1}\frac{d\tau}{\tau}\right]\frac{ds}{s}\\
&=& |c_0|\mathbb{E}_{\alpha,\alpha}(\lambda)+\int_{1}^{t}\left(\ln \frac{t}{s}\right)^{\alpha-1}\mathbb{E}_{\alpha,\alpha}\left(\lambda\left(\ln \frac{t}{s}\right)^{\alpha}\right)\left[\frac{L_1\|x\|_{\gamma,\ln}}{\Gamma(\beta)}(\ln s)^{\beta-\gamma}\mathbb{B}(\beta,1-\gamma)+\frac{L_2}{\Gamma(\beta+1)}(\ln s)^{\beta}\right.\\
&+&\left.\left(\frac{L_1\|x\|_{\gamma,\ln}}{\Gamma(\beta)}\mathbb{B}(\beta,1-\gamma)+\frac{L_2}{\Gamma(\beta+1)}\right)(\ln s)^{\beta-1}\right]\frac{ds}{s}\\
&\leq& |c_0|\mathbb{E}_{\alpha,\alpha}(\lambda)+\frac{L_1\|x\|_{\gamma,\ln}}{\Gamma(\beta)}\mathbb{B}(\beta,1-\gamma)\mathbb{E}_{\alpha,\alpha}(\lambda)\int_{1}^{t}\left(\ln \frac{t}{s}\right)^{\alpha-1}(\ln s)^{\beta-\gamma}\frac{ds}{s}\\
&+&\frac{L_2}{\Gamma(\beta+1)}\mathbb{E}_{\alpha,\alpha}(\lambda)\int_{1}^{t}\left(\ln \frac{t}{s}\right)^{\alpha-1}(\ln s)^{\beta}\frac{ds}{s}\\
&+&\left(\frac{L_1\|x\|_{\gamma,\ln}}{\Gamma(\beta)}\mathbb{B}(\beta,1-\gamma)+\frac{L_2}{\Gamma(\beta+1)}\right)\mathbb{E}_{\alpha,\alpha}(\lambda)\int_{1}^{t}\left(\ln \frac{t}{s}\right)^{\alpha-1}(\ln s)^{\beta-1}\frac{ds}{s}\\
&=&|c_0|\mathbb{E}_{\alpha,\alpha}(\lambda)+\frac{L_1\|x\|_{\gamma,\ln}}{\Gamma(\beta)}\mathbb{B}(\beta,1-\gamma)\mathbb{E}_{\alpha,\alpha}(\lambda)(\ln t)^{\alpha+\beta-\gamma}\mathbb{B}(\alpha,1+\beta-\gamma)\\
&+&\frac{L_2}{\Gamma(\beta+1)}\mathbb{E}_{\alpha,\alpha}(\lambda)(\ln t)^{\alpha+\beta}\mathbb{B}(\alpha,1+\beta)\\
&+&\left(\frac{L_1\|x\|_{\gamma,\ln}}{\Gamma(\beta)}\mathbb{B}(\beta,1-\gamma)+\frac{L_2}{\Gamma(\beta+1)}\right)\mathbb{E}_{\alpha,\alpha}(\lambda)(\ln t)^{\alpha+\beta-1}\mathbb{B}(\alpha,\beta)\\
&\leq&\omega_1+L_1\omega_2\|x\|_{\gamma,\ln}\\
&\leq&\omega_1+L_1\omega_2 r\leq r,
\end{eqnarray*}
which implies that
$\|\mathcal{H}x\|_{\gamma,\ln}\leq r$, therefore $\mathcal{H}(B_{r})\subset B_{r}$.\\

\textbf{Step 2.} We show that $\mathcal{H}$ is continuous.\\

$\left|(\ln t)^{\gamma}((\mathcal{H}x_n)(t)-(\mathcal{H}x)(t))\right|$
\begin{eqnarray*}
&\leq&\left|(\ln t)^{\gamma}\int_{1}^{t}\left(\ln \frac{t}{s}\right)^{\alpha-1}\mathbb{E}_{\alpha,\alpha}\left(\lambda\left(\ln \frac{t}{s}\right)^{\alpha}\right)\left[\frac{1}{\Gamma(\beta)}\int_{1}^{s}\left(\ln \frac{s}{\tau}\right)^{\beta-1}(f(\tau,x_n(\tau))-f(\tau,x(\tau)))\frac{d\tau}{\tau}\right.\right.\\
&-&\left.\left.\frac{(\ln s)^{\beta-1}}{\Gamma(\beta)}\int_{1}^{e}\left(1-\ln\tau\right)^{\beta-1}(f(\tau,x_n(\tau))-f(\tau,x(\tau)))\frac{d\tau}{\tau}\right]\frac{ds}{s}\right|\\
&\leq&\mathbb{E}_{\alpha,\alpha}(\lambda)\int_{1}^{t}\left(\ln \frac{t}{s}\right)^{\alpha-1}\left[\frac{1}{\Gamma(\beta)}\int_{1}^{s}\left(\ln \frac{s}{\tau}\right)^{\beta-1}\big|f(\tau,x_n(\tau))-f(\tau,x(\tau))\big|\frac{d\tau}{\tau}\right.\\
&+&\left.\frac{(\ln s)^{\beta-1}}{\Gamma(\beta)}\int_{1}^{e}\left(1-\ln\tau\right)^{\beta-1}\big|f(\tau,x_n(\tau))-f(\tau,x(\tau))\big|\frac{d\tau}{\tau}\right]\frac{ds}{s}
\end{eqnarray*}
\begin{eqnarray*}
&\leq&\mathbb{E}_{\alpha,\alpha}(\lambda)\int_{1}^{t}\left(\ln \frac{t}{s}\right)^{\alpha-1}\left[\frac{1}{\Gamma(\beta)}\int_{1}^{s}\left(\ln \frac{s}{\tau}\right)^{\beta-1}(\ln \tau)^{-\gamma}\big\|f(\cdot,x_n(\cdot))-f(\cdot,x(\cdot))\big\|_{\gamma,\ln}\frac{d\tau}{\tau}\right.\\
&+&\left.\frac{(\ln s)^{\beta-1}}{\Gamma(\beta)}\int_{1}^{e}\left(1-\ln\tau\right)^{\beta-1}(\ln \tau)^{-\gamma}\big\|f(\cdot,x_n(\cdot))-f(\cdot,x(\cdot))\big\|_{\gamma,\ln}\frac{d\tau}{\tau}\right]\frac{ds}{s}\\
&=&\frac{\mathbb{B}(\beta,1-\gamma)}{\Gamma(\beta)}\mathbb{E}_{\alpha,\alpha}(\lambda)\left[\int_{1}^{t}\left(\ln \frac{t}{s}\right)^{\alpha-1}(\ln s)^{\beta-\gamma}\frac{ds}{s}\right.\\
&+&\left.\int_{1}^{t}\left(\ln \frac{t}{s}\right)^{\alpha-1}(\ln s)^{\beta-1}\frac{ds}{s}\right]\big\|f(\cdot,x_n(\cdot))-f(\cdot,x(\cdot))\big\|_{\gamma,\ln}\\
&=&\frac{\mathbb{B}(\beta,1-\gamma)}{\Gamma(\beta)}\mathbb{E}_{\alpha,\alpha}(\lambda)\left[(\ln t)^{\alpha+\beta-\gamma}\mathbb{B}(\alpha,1+\beta-\gamma)+(\ln t)^{\alpha+\beta-1}\mathbb{B}(\alpha,\beta)\right]\big\|f(\cdot,x_n(\cdot))-f(\cdot,x(\cdot))\big\|_{\gamma,\ln}\\
&\leq&\frac{\mathbb{B}(\beta,1-\gamma)}{\Gamma(\beta)}\bigg(\mathbb{B}(\alpha,1+\beta-\gamma)+\mathbb{B}(\alpha,\beta)\bigg)\mathbb{E}_{\alpha,\alpha}(\lambda)\big\|f(\cdot,x_n(\cdot))-f(\cdot,x(\cdot))\big\|_{\gamma,\ln}.
\end{eqnarray*}
Hence, we get
$$\big\|\mathcal{H}x_n-\mathcal{H}x\big\|_{\gamma,\ln}\leq \omega_2\big\|f(\cdot,x_n(\cdot))-f(\cdot,x(\cdot))\big\|_{\gamma,\ln}~~,$$
and the continuity of $f$ implies that $\mathcal{H}$ is continuous.\\

\textbf{Step 3.}~ We show $\mathcal{H}$ is relatively compact on $B_r$.\\

According to \textbf{Step 1}, we showed that $\mathcal{H}(B_{r})\subset B_{r}$. Thus $\mathcal{H}(B_{r})$ is uniformly bounded.
 It remains to show that $\mathcal{H}$ is equicontinuous.\\

For $1<t_1<t_2\leq e$ and $x\in B_r$, we have\\

$\big|(\mathcal{H}x)(t_2)-(\mathcal{H}x)(t_1)\big|$
\begin{eqnarray*}
&\leq&\big|c_{0}(\ln t_2)^{\alpha-1}\mathbb{E}_{\alpha,\alpha}\big(\lambda(\ln t_2)^{\alpha}\big)-c_{0}(\ln t_1)^{\alpha-1}\mathbb{E}_{\alpha,\alpha}\big(\lambda(\ln t_2)^{\alpha}\big)\big|\\
&+&\big|c_{0}(\ln t_1)^{\alpha-1}\mathbb{E}_{\alpha,\alpha}\big(\lambda(\ln t_2)^{\alpha}\big)-c_{0}(\ln t_1)^{\alpha-1}\mathbb{E}_{\alpha,\alpha}\big(\lambda(\ln t_1)^{\alpha}\big)\big|\\
&+&\bigg|\int_{1}^{t_2}\left(\ln \frac{t_2}{s}\right)^{\alpha-1}\mathbb{E}_{\alpha,\alpha}\big(\lambda\left(\ln \frac{t_2}{s}\right)^{\alpha}\big)\bigg[\frac{1}{\Gamma(\beta)}\int_{1}^{s}\left(\ln \frac{s}{\tau}\right)^{\beta-1}f(\tau,x(\tau))\frac{d\tau}{\tau}\\
&-&\frac{(\ln s)^{\beta-1}}{\Gamma(\beta)}\int_{1}^{e}\left(1-\ln\tau\right)^{\beta-1}f(\tau,x(\tau))\frac{d\tau}{\tau}\bigg]\frac{ds}{s}\\
&-&\int_{1}^{t_1}\left(\ln \frac{t_2}{s}\right)^{\alpha-1}\mathbb{E}_{\alpha,\alpha}\big(\lambda\left(\ln \frac{t_1}{s}\right)^{\alpha}\big)\bigg[\frac{1}{\Gamma(\beta)}\int_{1}^{s}\left(\ln \frac{s}{\tau}\right)^{\beta-1}f(\tau,x(\tau))\frac{d\tau}{\tau}\\
&-&\frac{(\ln s)^{\beta-1}}{\Gamma(\beta)}\int_{1}^{e}\left(1-\ln\tau\right)^{\beta-1}f(\tau,x(\tau))\frac{d\tau}{\tau}\bigg]\frac{ds}{s}\bigg|
\end{eqnarray*}
\begin{eqnarray*}
&\leq&\bigg|c_{0}\mathbb{E}_{\alpha,\alpha}(\lambda(\ln t_2)^{\alpha})\bigg((\ln t_2)^{\alpha-1}-(\ln t_1)^{\alpha-1}\bigg)\bigg|\\
&+&\bigg|c_{0}(\ln t_1)^{\alpha-1}\bigg(\mathbb{E}_{\alpha,\alpha}\big(\lambda(\ln t_2)^{\alpha})-\mathbb{E}_{\alpha,\alpha}(\lambda(\ln t_1)^{\alpha}\big)\bigg)\bigg|\\
&+&\bigg|\int_{1}^{t_2}\left(\ln \frac{t_2}{s}\right)^{\alpha-1}\mathbb{E}_{\alpha,\alpha}\big(\lambda\left(\ln \frac{t_2}{s}\right)^{\alpha}\big)\bigg[\frac{1}{\Gamma(\beta)}\int_{1}^{s}\left(\ln \frac{s}{\tau}\right)^{\beta-1}f(\tau,x(\tau))\frac{d\tau}{\tau}\\
&-&\frac{(\ln s)^{\beta-1}}{\Gamma(\beta)}\int_{1}^{e}\left(1-\ln\tau\right)^{\beta-1}f(\tau,x(\tau))\frac{d\tau}{\tau}\bigg]\frac{ds}{s}\\
&-&\int_{1}^{t_2}\left(\ln \frac{t_2}{s}\right)^{\alpha-1}\mathbb{E}_{\alpha,\alpha}\big(\lambda\left(\ln \frac{t_1}{s}\right)^{\alpha}\big)\bigg[\frac{1}{\Gamma(\beta)}\int_{1}^{s}\left(\ln \frac{s}{\tau}\right)^{\beta-1}f(\tau,x(\tau))\frac{d\tau}{\tau}\\
&-&\frac{(\ln s)^{\beta-1}}{\Gamma(\beta)}\int_{1}^{e}\left(1-\ln\tau\right)^{\beta-1}f(\tau,x(\tau))\frac{d\tau}{\tau}\bigg]\frac{ds}{s}\bigg|\\
&+&\bigg|\int_{1}^{t_2}\left(\ln \frac{t_2}{s}\right)^{\alpha-1}\mathbb{E}_{\alpha,\alpha}\big(\lambda\left(\ln \frac{t_1}{s}\right)^{\alpha}\big)\bigg[\frac{1}{\Gamma(\beta)}\int_{1}^{s}\left(\ln \frac{s}{\tau}\right)^{\beta-1}f(\tau,x(\tau))\frac{d\tau}{\tau}\\
&-&\frac{(\ln s)^{\beta-1}}{\Gamma(\beta)}\int_{1}^{e}\left(1-\ln\tau\right)^{\beta-1}f(\tau,x(\tau))\frac{d\tau}{\tau}\bigg]\frac{ds}{s}\\
&-&\int_{1}^{t_2}\left(\ln \frac{t_1}{s}\right)^{\alpha-1}\mathbb{E}_{\alpha,\alpha}\big(\lambda\left(\ln \frac{t_1}{s}\right)^{\alpha}\big)\bigg[\frac{1}{\Gamma(\beta)}\int_{1}^{s}\left(\ln \frac{s}{\tau}\right)^{\beta-1}f(\tau,x(\tau))\frac{d\tau}{\tau}\\
&-&\frac{(\ln s)^{\beta-1}}{\Gamma(\beta)}\int_{1}^{e}\left(1-\ln\tau\right)^{\beta-1}f(\tau,x(\tau))\frac{d\tau}{\tau}\bigg]\frac{ds}{s}\bigg|\\
&+&\bigg|\int_{t_1}^{t_2}\left(\ln \frac{t_1}{s}\right)^{\alpha-1}\mathbb{E}_{\alpha,\alpha}\big(\lambda\left(\ln \frac{t_1}{s}\right)^{\alpha}\big)\bigg[\frac{1}{\Gamma(\beta)}\int_{1}^{s}\left(\ln \frac{s}{\tau}\right)^{\beta-1}f(\tau,x(\tau))\frac{d\tau}{\tau}\\
&-&\frac{(\ln s)^{\beta-1}}{\Gamma(\beta)}\int_{1}^{e}\left(1-\ln\tau\right)^{\beta-1}f(\tau,x(\tau))\frac{d\tau}{\tau}\bigg]\frac{ds}{s}\bigg|\\
&\leq&\mathbb{E}_{\alpha,\alpha}(\lambda)\bigg|c_{0}\bigg((\ln t_2)^{\alpha-1}-(\ln t_1)^{\alpha-1}\bigg)\bigg|\\
&+&\bigg|c_{0}(\ln t_1)^{\alpha-1}\bigg(\mathbb{E}_{\alpha,\alpha}\big(\lambda(\ln t_2)^{\alpha})-\mathbb{E}_{\alpha,\alpha}(\lambda(\ln t_1)^{\alpha}\big)\bigg)\bigg|\\
&+&\bigg|\int_{1}^{t_2}\left(\ln \frac{t_2}{s}\right)^{\alpha-1}\bigg(\mathbb{E}_{\alpha,\alpha}\big(\lambda\left(\ln \frac{t_2}{s}\right)^{\alpha}\big)-\mathbb{E}_{\alpha,\alpha}\big(\lambda\left(\ln \frac{t_1}{s}\right)^{\alpha}\big)\bigg)
\bigg[\frac{1}{\Gamma(\beta)}\int_{1}^{s}\left(\ln \frac{s}{\tau}\right)^{\beta-1}f(\tau,x(\tau))\frac{d\tau}{\tau}\\
&-&\frac{(\ln s)^{\beta-1}}{\Gamma(\beta)}\int_{1}^{e}\left(1-\ln\tau\right)^{\beta-1}f(\tau,x(\tau))\frac{d\tau}{\tau}\bigg]\frac{ds}{s}\bigg|\\
&+&\bigg|\int_{1}^{t_2}\mathbb{E}_{\alpha,\alpha}\big(\lambda\left(\ln \frac{t_1}{s}\right)^{\alpha}\bigg(\left(\ln \frac{t_2}{s}\right)^{\alpha-1}-\left(\ln \frac{t_1}{s}\right)^{\alpha-1}\bigg)
\bigg[\frac{1}{\Gamma(\beta)}\int_{1}^{s}\left(\ln \frac{s}{\tau}\right)^{\beta-1}f(\tau,x(\tau))\frac{d\tau}{\tau}\\
&-&\frac{(\ln s)^{\beta-1}}{\Gamma(\beta)}\int_{1}^{e}\left(1-\ln\tau\right)^{\beta-1}f(\tau,x(\tau))\frac{d\tau}{\tau}\bigg]\frac{ds}{s}\bigg|\\
&+&\bigg|\int_{t_1}^{t_2}\left(\ln \frac{t_1}{s}\right)^{\alpha-1}\mathbb{E}_{\alpha,\alpha}\big(\lambda\left(\ln \frac{t_1}{s}\right)^{\alpha}\big)\bigg[\frac{1}{\Gamma(\beta)}\int_{1}^{s}\left(\ln \frac{s}{\tau}\right)^{\beta-1}f(\tau,x(\tau))\frac{d\tau}{\tau}\\
&-&\frac{(\ln s)^{\beta-1}}{\Gamma(\beta)}\int_{1}^{e}\left(1-\ln\tau\right)^{\beta-1}f(\tau,x(\tau))\frac{d\tau}{\tau}\bigg]\frac{ds}{s}\bigg|
\end{eqnarray*}
\begin{eqnarray*}
&\leq&|c_0|\mathbb{E}_{\alpha,\alpha}(\lambda)\bigg((\alpha-1)\frac{(\ln \zeta)^{\alpha-2}}{\zeta}|t_2-t_1|\bigg)+|c_0|\bigg(\mathcal{O}|t_2-t_1|\bigg)\\
&+&\int_{1}^{t_2}\left(\ln \frac{t_2}{s}\right)^{\alpha-1}\bigg(\mathbb{E}_{\alpha,\alpha}\big(\lambda\left(\ln \frac{t_2}{s}\right)^{\alpha}\big)-\mathbb{E}_{\alpha,\alpha}\big(\lambda\left(\ln \frac{t_1}{s}\right)^{\alpha}\big)\bigg)
\bigg[\frac{1}{\Gamma(\beta)}\int_{1}^{s}\left(\ln \frac{s}{\tau}\right)^{\beta-1}|f(\tau,x(\tau))|\frac{d\tau}{\tau}\\
&+&\frac{(\ln s)^{\beta-1}}{\Gamma(\beta)}\int_{1}^{e}\left(1-\ln\tau\right)^{\beta-1}|f(\tau,x(\tau))|\frac{d\tau}{\tau}\bigg]\frac{ds}{s}\\
&+&\int_{1}^{t_2}\mathbb{E}_{\alpha,\alpha}\big(\lambda\left(\ln \frac{t_1}{s}\right)^{\alpha}\bigg(\left(\ln \frac{t_2}{s}\right)^{\alpha-1}-\left(\ln \frac{t_1}{s}\right)^{\alpha-1}\bigg)
\bigg[\frac{1}{\Gamma(\beta)}\int_{1}^{s}\left(\ln \frac{s}{\tau}\right)^{\beta-1}|f(\tau,x(\tau))|\frac{d\tau}{\tau}\\
&+&\frac{(\ln s)^{\beta-1}}{\Gamma(\beta)}\int_{1}^{e}\left(1-\ln\tau\right)^{\beta-1}|f(\tau,x(\tau))|\frac{d\tau}{\tau}\bigg]\frac{ds}{s}\\
&+&\int_{t_1}^{t_2}\left(\ln \frac{t_1}{s}\right)^{\alpha-1}\mathbb{E}_{\alpha,\alpha}\big(\lambda\left(\ln \frac{t_1}{s}\right)^{\alpha}\big)\bigg[\frac{1}{\Gamma(\beta)}\int_{1}^{s}\left(\ln \frac{s}{\tau}\right)^{\beta-1}|f(\tau,x(\tau))|\frac{d\tau}{\tau}\\
&+&\frac{(\ln s)^{\beta-1}}{\Gamma(\beta)}\int_{1}^{e}\left(1-\ln\tau\right)^{\beta-1}|f(\tau,x(\tau))|\frac{d\tau}{\tau}\bigg]\frac{ds}{s}\\
&\leq&|c_0|\mathbb{E}_{\alpha,\alpha}(\lambda)\bigg((\alpha-1)\frac{(\ln \zeta)^{\alpha-2}}{\zeta}|t_2-t_1|\bigg)+|c_0|\bigg(\mathcal{O}|t_2-t_1|\bigg)\\
&+&\int_{1}^{t_2}\left(\ln \frac{t_2}{s}\right)^{\alpha-1}\bigg(\mathbb{E}_{\alpha,\alpha}\big(\lambda\left(\ln \frac{t_2}{s}\right)^{\alpha}\big)-\mathbb{E}_{\alpha,\alpha}\big(\lambda\left(\ln \frac{t_1}{s}\right)^{\alpha}\big)\bigg)
\bigg[\frac{1}{\Gamma(\beta)}\int_{1}^{s}\left(\ln \frac{s}{\tau}\right)^{\beta-1}(L_1 |x(\tau)|+L_2)\frac{d\tau}{\tau}\\
&+&\frac{(\ln s)^{\beta-1}}{\Gamma(\beta)}\int_{1}^{e}\left(1-\ln\tau\right)^{\beta-1}(L_1 |x(\tau)|+L_2)\frac{d\tau}{\tau}\bigg]\frac{ds}{s}\\
&+&\int_{1}^{t_2}\mathbb{E}_{\alpha,\alpha}\big(\lambda\left(\ln \frac{t_1}{s}\right)^{\alpha}\bigg(\left(\ln \frac{t_2}{s}\right)^{\alpha-1}-\left(\ln \frac{t_1}{s}\right)^{\alpha-1}\bigg)
\bigg[\frac{1}{\Gamma(\beta)}\int_{1}^{s}\left(\ln \frac{s}{\tau}\right)^{\beta-1}(L_1 |x(\tau)|+L_2)\frac{d\tau}{\tau}\\
&+&\frac{(\ln s)^{\beta-1}}{\Gamma(\beta)}\int_{1}^{e}\left(1-\ln\tau\right)^{\beta-1}(L_1 |x(\tau)|+L_2)\frac{d\tau}{\tau}\bigg]\frac{ds}{s}\\
&+&\int_{t_1}^{t_2}\left(\ln \frac{t_1}{s}\right)^{\alpha-1}\mathbb{E}_{\alpha,\alpha}\big(\lambda\left(\ln \frac{t_1}{s}\right)^{\alpha}\big)\bigg[\frac{1}{\Gamma(\beta)}\int_{1}^{s}\left(\ln \frac{s}{\tau}\right)^{\beta-1}(L_1 |x(\tau)|+L_2)\frac{d\tau}{\tau}\\
&+&\frac{(\ln s)^{\beta-1}}{\Gamma(\beta)}\int_{1}^{e}\left(1-\ln\tau\right)^{\beta-1}(L_1 |x(\tau)|+L_2)\frac{d\tau}{\tau}\bigg]\frac{ds}{s}
\end{eqnarray*}
\begin{eqnarray*}
&\leq&|c_0|\mathbb{E}_{\alpha,\alpha}(\lambda)\bigg((\alpha-1)\frac{(\ln \zeta)^{\alpha-2}}{\zeta}|t_2-t_1|\bigg)+|c_0|\bigg(\mathcal{O}|t_2-t_1|\bigg)\\
&+&\int_{1}^{t_2}\left(\ln \frac{t_2}{s}\right)^{\alpha-1}\bigg(\mathbb{E}_{\alpha,\alpha}\big(\lambda\left(\ln \frac{t_2}{s}\right)^{\alpha}\big)-\mathbb{E}_{\alpha,\alpha}\big(\lambda\left(\ln \frac{t_1}{s}\right)^{\alpha}\big)\bigg)\\
&&\times\bigg[\frac{L_1}{\Gamma(\beta)}\int_{1}^{s}\left(\ln \frac{s}{\tau}\right)^{\beta-1}(\ln \tau)^{-\gamma}\|x\|_{\gamma,\ln}\frac{d\tau}{\tau}+\frac{L_2}{\Gamma(\beta)}\int_{1}^{s}\left(\ln \frac{s}{\tau}\right)^{\beta-1}\frac{d\tau}{\tau}\\
&&+\frac{L_1(\ln s)^{\beta-1}}{\Gamma(\beta)}\int_{1}^{e}\left(1-\ln\tau\right)^{\beta-1}(\ln \tau)^{-\gamma}\|x\|_{\gamma,\ln}\frac{d\tau}{\tau}+\frac{L_2(\ln s)^{\beta-1}}{\Gamma(\beta)}\int_{1}^{e}\left(1-\ln\tau\right)^{\beta-1}\frac{d\tau}{\tau}\bigg]\frac{ds}{s}\\
&+&\int_{1}^{t_2}\mathbb{E}_{\alpha,\alpha}\big(\lambda\left(\ln \frac{t_1}{s}\right)^{\alpha}\bigg(\left(\ln \frac{t_2}{s}\right)^{\alpha-1}-\left(\ln \frac{t_1}{s}\right)^{\alpha-1}\bigg)\bigg[\frac{L_1}{\Gamma(\beta)}\int_{1}^{s}\left(\ln \frac{s}{\tau}\right)^{\beta-1}(\ln \tau)^{-\gamma}\|x\|_{\gamma,\ln}\frac{d\tau}{\tau}+\frac{L_2}{\Gamma(\beta)}\int_{1}^{s}\left(\ln \frac{s}{\tau}\right)^{\beta-1}\frac{d\tau}{\tau}\\
&+&\frac{L_1(\ln s)^{\beta-1}}{\Gamma(\beta)}\int_{1}^{e}\left(1-\ln\tau\right)^{\beta-1}(\ln \tau)^{-\gamma}\|x\|_{\gamma,\ln}\frac{d\tau}{\tau}+\frac{L_2(\ln s)^{\beta-1}}{\Gamma(\beta)}\int_{1}^{e}\left(1-\ln\tau\right)^{\beta-1}\frac{d\tau}{\tau}\bigg]\frac{ds}{s}\\
&+&\int_{t_1}^{t_2}\left(\ln \frac{t_1}{s}\right)^{\alpha-1}\mathbb{E}_{\alpha,\alpha}\big(\lambda\left(\ln \frac{t_1}{s}\right)^{\alpha}\big)
\bigg[\frac{L_1}{\Gamma(\beta)}\int_{1}^{s}\left(\ln \frac{s}{\tau}\right)^{\beta-1}(\ln \tau)^{-\gamma}\|x\|_{\gamma,\ln}\frac{d\tau}{\tau}+\frac{L_2}{\Gamma(\beta)}\int_{1}^{s}\left(\ln \frac{s}{\tau}\right)^{\beta-1}\frac{d\tau}{\tau}\\
&+&\frac{L_1(\ln s)^{\beta-1}}{\Gamma(\beta)}\int_{1}^{e}\left(1-\ln\tau\right)^{\beta-1}(\ln \tau)^{-\gamma}\|x\|_{\gamma,\ln}\frac{d\tau}{\tau}+\frac{L_2(\ln s)^{\beta-1}}{\Gamma(\beta)}\int_{1}^{e}\left(1-\ln\tau\right)^{\beta-1}\frac{d\tau}{\tau}\bigg]\frac{ds}{s}\\
&=&|c_0|\mathbb{E}_{\alpha,\alpha}(\lambda)\bigg((\alpha-1)\frac{(\ln \zeta)^{\alpha-2}}{\zeta}|t_2-t_1|\bigg)+|c_0|\bigg(\mathcal{O}|t_2-t_1|\bigg)\\
&+&\int_{1}^{t_2}\left(\ln \frac{t_2}{s}\right)^{\alpha-1}\bigg(\mathbb{E}_{\alpha,\alpha}\big(\lambda\left(\ln \frac{t_2}{s}\right)^{\alpha}\big)-\mathbb{E}_{\alpha,\alpha}\big(\lambda\left(\ln \frac{t_1}{s}\right)^{\alpha}\big)\bigg)\bigg[\frac{L_1 r}{\Gamma(\beta)}(\ln s)^{\beta-\gamma}\mathbb{B}(\beta,1-\gamma)+\frac{L_2}{\Gamma(\beta+1)}(\ln s)^{\beta}\\
&+&\left(\frac{L_1 r}{\Gamma(\beta)}\mathbb{B}(\beta,1-\gamma)+\frac{L_2}{\Gamma(\beta+1)}\right)(\ln s)^{\beta-1}\bigg]\frac{ds}{s}\\
&+&\int_{1}^{t_2}\mathbb{E}_{\alpha,\alpha}\big(\lambda\left(\ln \frac{t_1}{s}\right)^{\alpha}\bigg(\left(\ln \frac{t_2}{s}\right)^{\alpha-1}-\left(\ln \frac{t_1}{s}\right)^{\alpha-1}\bigg)
\bigg[\frac{L_1 r}{\Gamma(\beta)}(\ln s)^{\beta-\gamma}\mathbb{B}(\beta,1-\gamma)+\frac{L_2}{\Gamma(\beta+1)}(\ln s)^{\beta}\\
&+&\left(\frac{L_1 r}{\Gamma(\beta)}\mathbb{B}(\beta,1-\gamma)+\frac{L_2}{\Gamma(\beta+1)}\right)(\ln s)^{\beta-1}\bigg]\frac{ds}{s}\\
&+&\int_{t_1}^{t_2}\left(\ln \frac{t_1}{s}\right)^{\alpha-1}\mathbb{E}_{\alpha,\alpha}\big(\lambda\left(\ln \frac{t_1}{s}\right)^{\alpha}\big)\bigg[\frac{L_1 r}{\Gamma(\beta)}(\ln s)^{\beta-\gamma}\mathbb{B}(\beta,1-\gamma)+\frac{L_2}{\Gamma(\beta+1)}(\ln s)^{\beta}\\
&+&\left(\frac{L_1 r}{\Gamma(\beta)}\mathbb{B}(\beta,1-\gamma)+\frac{L_2}{\Gamma(\beta+1)}\right)(\ln s)^{\beta-1}\bigg]\frac{ds}{s}\\
&=&|c_0|\mathbb{E}_{\alpha,\alpha}(\lambda)\bigg((\alpha-1)\frac{(\ln \zeta)^{\alpha-2}}{\zeta}|t_2-t_1|\bigg)+|c_0|\bigg(\mathcal{O}|t_2-t_1|\bigg)+I_1+I_2+I_3,
\end{eqnarray*}
where $\zeta\in (t_1,t_2)$ and
$$I_1=\int_{1}^{t_2}\left(\ln \frac{t_2}{s}\right)^{\alpha-1}\bigg(\mathbb{E}_{\alpha,\alpha}\big(\lambda\left(\ln \frac{t_2}{s}\right)^{\alpha}\big)-\mathbb{E}_{\alpha,\alpha}\big(\lambda\left(\ln \frac{t_1}{s}\right)^{\alpha}\big)\bigg)\bigg[\frac{L_1 r}{\Gamma(\beta)}(\ln s)^{\beta-\gamma}\mathbb{B}(\beta,1-\gamma)$$
$$+\frac{L_2}{\Gamma(\beta+1)}(\ln s)^{\beta}+\left(\frac{L_1 r}{\Gamma(\beta)}\mathbb{B}(\beta,1-\gamma)+\frac{L_2}{\Gamma(\beta+1)}\right)(\ln s)^{\beta-1}\bigg]\frac{ds}{s},$$

$$I_2=\int_{1}^{t_2}\mathbb{E}_{\alpha,\alpha}\big(\lambda\left(\ln \frac{t_1}{s}\right)^{\alpha}\bigg(\left(\ln \frac{t_2}{s}\right)^{\alpha-1}-\left(\ln \frac{t_1}{s}\right)^{\alpha-1}\bigg)
\bigg[\frac{L_1 r}{\Gamma(\beta)}(\ln s)^{\beta-\gamma}\mathbb{B}(\beta,1-\gamma)+\frac{L_2}{\Gamma(\beta+1)}(\ln s)^{\beta}$$
$$+\left(\frac{L_1 r}{\Gamma(\beta)}\mathbb{B}(\beta,1-\gamma)+\frac{L_2}{\Gamma(\beta+1)}\right)(\ln s)^{\beta-1}\bigg]\frac{ds}{s},$$

$$I_3=\int_{t_1}^{t_2}\left(\ln \frac{t_1}{s}\right)^{\alpha-1}\mathbb{E}_{\alpha,\alpha}\big(\lambda\left(\ln \frac{t_1}{s}\right)^{\alpha}\big)\bigg[\frac{L_1 r}{\Gamma(\beta)}(\ln s)^{\beta-\gamma}\mathbb{B}(\beta,1-\gamma)+\frac{L_2}{\Gamma(\beta+1)}(\ln s)^{\beta}$$
$$+\left(\frac{L_1 r}{\Gamma(\beta)}\mathbb{B}(\beta,1-\gamma)+\frac{L_2}{\Gamma(\beta+1)}\right)(\ln s)^{\beta-1}\bigg]\frac{ds}{s}.$$
For $\frac{1}{p}+\frac{1}{q}=1,~p,q>1$, we have the following estimations.
\begin{eqnarray*}
I_1&=&\int_{1}^{t_2}\left(\ln \frac{t_2}{s}\right)^{\alpha-1}\bigg(\mathbb{E}_{\alpha,\alpha}\big(\lambda\left(\ln \frac{t_2}{s}\right)^{\alpha}\big)-\mathbb{E}_{\alpha,\alpha}\big(\lambda\left(\ln \frac{t_1}{s}\right)^{\alpha}\big)\bigg)\bigg[\frac{L_1 r}{\Gamma(\beta)}(\ln s)^{\beta-\gamma}\mathbb{B}(\beta,1-\gamma)\\
&+&\frac{L_2}{\Gamma(\beta+1)}(\ln s)^{\beta}+\left(\frac{L_1 r}{\Gamma(\beta)}\mathbb{B}(\beta,1-\gamma)+\frac{L_2}{\Gamma(\beta+1)}\right)(\ln s)^{\beta-1}\bigg]\frac{ds}{s}\\
&=&\frac{L_1 r}{\Gamma(\beta)}\mathbb{B}(\beta,1-\gamma)\bigg[\int_{1}^{t_2}\left(\ln \frac{t_2}{s}\right)^{\alpha-1}(\ln s)^{\beta-\gamma}\bigg(\mathbb{E}_{\alpha,\alpha}\big(\lambda\left(\ln \frac{t_2}{s}\right)^{\alpha}\big)-\mathbb{E}_{\alpha,\alpha}\big(\lambda\left(\ln \frac{t_1}{s}\right)^{\alpha}\big)\bigg)\frac{ds}{s}\\
&+&\int_{1}^{t_2}\left(\ln \frac{t_2}{s}\right)^{\alpha-1}(\ln s)^{\beta-1}\bigg(\mathbb{E}_{\alpha,\alpha}\big(\lambda\left(\ln \frac{t_2}{s}\right)^{\alpha}\big)-\mathbb{E}_{\alpha,\alpha}\big(\lambda\left(\ln \frac{t_1}{s}\right)^{\alpha}\big)\bigg)\frac{ds}{s}\bigg]\\
&+&\frac{L_2}{\Gamma(\beta+1)}\bigg[\int_{1}^{t_2}\left(\ln \frac{t_2}{s}\right)^{\alpha-1}(\ln s)^{\beta}\bigg(\mathbb{E}_{\alpha,\alpha}\big(\lambda\left(\ln \frac{t_2}{s}\right)^{\alpha}\big)-\mathbb{E}_{\alpha,\alpha}\big(\lambda\left(\ln \frac{t_1}{s}\right)^{\alpha}\big)\bigg)\frac{ds}{s}\\
&+&\int_{1}^{t_2}\left(\ln \frac{t_2}{s}\right)^{\alpha-1}(\ln s)^{\beta-1}\bigg(\mathbb{E}_{\alpha,\alpha}\big(\lambda\left(\ln \frac{t_2}{s}\right)^{\alpha}\big)-\mathbb{E}_{\alpha,\alpha}\big(\lambda\left(\ln \frac{t_1}{s}\right)^{\alpha}\big)\bigg)\frac{ds}{s}\bigg]\\
&\leq&\frac{L_1 r}{\Gamma(\beta)}\mathbb{B}(\beta,1-\gamma)\bigg[\bigg(\int_{1}^{t_2}\left(\ln \frac{t_2}{s}\right)^{p(\alpha-1)}(\ln s)^{p(\beta-\gamma)}\frac{ds}{s}\bigg)^{\frac{1}{p}}\bigg(\int_{1}^{t_2}\bigg(\mathbb{E}_{\alpha,\alpha}\big(\lambda\left(\ln \frac{t_2}{s}\right)^{\alpha}\big)-\mathbb{E}_{\alpha,\alpha}\big(\lambda\left(\ln \frac{t_1}{s}\right)^{\alpha}\big)\bigg)^{q}\frac{ds}{s}\bigg)^{\frac{1}{q}}\\
&+&\bigg(\int_{1}^{t_2}\left(\ln \frac{t_2}{s}\right)^{p(\alpha-1)}(\ln s)^{p(\beta-1)}\frac{ds}{s}\bigg)^{\frac{1}{p}}\bigg(\int_{1}^{t_2}\bigg(\mathbb{E}_{\alpha,\alpha}\big(\lambda\left(\ln \frac{t_2}{s}\right)^{\alpha}\big)-\mathbb{E}_{\alpha,\alpha}\big(\lambda\left(\ln \frac{t_1}{s}\right)^{\alpha}\big)\bigg)^{q}\frac{ds}{s}\bigg)^{\frac{1}{q}}\bigg]\\
&+&\frac{L_2}{\Gamma(\beta+1)}\bigg[\bigg(\int_{1}^{t_2}\left(\ln \frac{t_2}{s}\right)^{p(\alpha-1)}(\ln s)^{p\beta}\frac{ds}{s}\bigg)^{\frac{1}{p}}\bigg(\int_{1}^{t_2}\bigg(\mathbb{E}_{\alpha,\alpha}\big(\lambda\left(\ln \frac{t_2}{s}\right)^{\alpha}\big)-\mathbb{E}_{\alpha,\alpha}\big(\lambda\left(\ln \frac{t_1}{s}\right)^{\alpha}\big)\bigg)^{q}\frac{ds}{s}\bigg)^{\frac{1}{q}}\\
&+&\bigg(\int_{1}^{t_2}\left(\ln \frac{t_2}{s}\right)^{p(\alpha-1)}(\ln s)^{p(\beta-1)}\frac{ds}{s}\bigg)^{\frac{1}{p}}\bigg(\int_{1}^{t_2}\bigg(\mathbb{E}_{\alpha,\alpha}\big(\lambda\left(\ln \frac{t_2}{s}\right)^{\alpha}\big)-\mathbb{E}_{\alpha,\alpha}\big(\lambda\left(\ln \frac{t_1}{s}\right)^{\alpha}\big)\bigg)^{q}\frac{ds}{s}\bigg)^{\frac{1}{q}}\bigg]
\end{eqnarray*}
\begin{eqnarray*}
&\leq&\frac{L_1 r}{\Gamma(\beta)}\mathbb{B}(\beta,1-\gamma)\bigg[\bigg(\mathbb{B}(p(\alpha-1)+1,p(\beta-\gamma)+1)\bigg)^{\frac{1}{p}}\bigg(\int_{1}^{t_2}\mathcal{O}(|t_2-t_1|)\frac{ds}{s}\bigg)^{\frac{1}{q}}\\
&+&\bigg(\mathbb{B}(p(\alpha-1)+1,p(\beta-1)+1)\bigg)^{\frac{1}{p}}\bigg(\int_{1}^{t_2}\mathcal{O}(|t_2-t_1|)\frac{ds}{s}\bigg)^{\frac{1}{q}}\bigg]\\
&+&\frac{L_2}{\Gamma(\beta+1)}\bigg[\bigg(\mathbb{B}(p(\alpha-1)+1,p\beta+1)\bigg)^{\frac{1}{p}}\bigg(\int_{1}^{t_2}\mathcal{O}(|t_2-t_1|)\frac{ds}{s}\bigg)^{\frac{1}{q}}\\
&&+\bigg(\mathbb{B}(p(\alpha-1)+1,p(\beta-1)+1)\bigg)^{\frac{1}{p}}\bigg(\int_{1}^{t_2}\mathcal{O}(|t_2-t_1|)\frac{ds}{s}\bigg)^{\frac{1}{q}}\bigg],
\end{eqnarray*}
which implies that $I_1\to 0$ as $t_1\to t_2$.\\

For $I_2$, using Lemma \ref{Ronglemma}, we have
\begin{eqnarray*}
I_2&\leq&\frac{L_1 r}{\Gamma(\beta)}\mathbb{B}(\beta,1-\gamma)\mathbb{E}_{\alpha,\alpha}(\lambda)\bigg[\bigg(\int_{1}^{t_2}(\ln s)^{p(\beta-\gamma)}\frac{ds}{s}\bigg)^{\frac{1}{p}}\bigg(\int_{1}^{t_2}\bigg(\left(\ln \frac{t_2}{s}\right)^{\alpha-1}-\left(\ln \frac{t_1}{s}\right)^{\alpha-1}\bigg)^{q}\frac{ds}{s}\bigg)^{\frac{1}{q}}\\
&+&\bigg(\int_{1}^{t_2}(\ln s)^{p(\beta-1)}\frac{ds}{s}\bigg)^{\frac{1}{p}}\bigg(\int_{1}^{t_2}\bigg(\left(\ln \frac{t_2}{s}\right)^{\alpha-1}-\left(\ln \frac{t_1}{s}\right)^{\alpha-1}\bigg)^{q}\frac{ds}{s}\bigg)^{\frac{1}{q}}\bigg]\\
&+&\frac{L_2}{\Gamma(\beta+1)}\mathbb{E}_{\alpha,\alpha}(\lambda)\bigg[\bigg(\int_{1}^{t_2}(\ln s)^{p\beta}\frac{ds}{s}\bigg)^{\frac{1}{p}}\bigg(\int_{1}^{t_2}\bigg(\left(\ln \frac{t_2}{s}\right)^{\alpha-1}-\left(\ln \frac{t_1}{s}\right)^{\alpha-1}\bigg)^{q}\frac{ds}{s}\bigg)^{\frac{1}{q}}\\
&+&\bigg(\int_{1}^{t_2}(\ln s)^{p(\beta-1)}\frac{ds}{s}\bigg)^{\frac{1}{p}}\bigg(\int_{1}^{t_2}\bigg(\left(\ln \frac{t_2}{s}\right)^{\alpha-1}-\left(\ln \frac{t_1}{s}\right)^{\alpha-1}\bigg)^{q}\frac{ds}{s}\bigg)^{\frac{1}{q}}\bigg]\\
&\leq&\Bigg\{\frac{L_1 r}{\Gamma(\beta)}\mathbb{B}(\beta,1-\gamma)\mathbb{E}_{\alpha,\alpha}(\lambda)\bigg[\bigg(\frac{(\ln t_2)^{p(\beta-\gamma)+1}}{p(\beta-\gamma)+1}\bigg)^{\frac{1}{p}}+\bigg(\frac{(\ln t_2)^{p(\beta-1)+1}}{p(\beta-1)+1}\bigg)^{\frac{1}{p}}\bigg]\\
&&+\frac{L_2}{\Gamma(\beta+1)}\mathbb{E}_{\alpha,\alpha}(\lambda)\bigg[\bigg(\frac{(\ln t_2)^{p\beta+1}}{p\beta+1}\bigg)^{\frac{1}{p}}+\bigg(\frac{(\ln t_2)^{p(\beta-1)+1}}{p(\beta-1)+1}\bigg)^{\frac{1}{p}}\bigg]\Bigg\}\\
&&\times\bigg(\frac{1}{q(\alpha-1)+1}\bigg((\ln t_2)^{q(\alpha-1)+1}-(\ln t_1)^{q(\alpha-1)+1}+\left(\ln \frac{t_1}{t_2}\right)^{q(\alpha-1)+1}\bigg)\bigg),
\end{eqnarray*}
which implies that $I_2\to 0$ as $t_1\to t_2$.\\

For $I_3$, using Lemma \ref{MVTh}, we have
\begin{eqnarray*}
I_3&=&\frac{L_1 r}{\Gamma(\beta)}\mathbb{B}(\beta,1-\gamma)\mathbb{E}_{\alpha,\alpha}(\lambda)\bigg[\left(\ln \frac{t_1}{\xi}\right)^{\alpha-1}\frac{(\ln \xi)^{\beta-\gamma}}{\xi}|t_2-t_1|+\left(\ln \frac{t_1}{\xi}\right)^{\alpha-1}\frac{(\ln \xi)^{\beta-1}}{\xi}|t_2-t_1|\bigg]\\
&+&\frac{L_2}{\Gamma(\beta+1)}\mathbb{E}_{\alpha,\alpha}(\lambda)\bigg[\left(\ln \frac{t_1}{\xi}\right)^{\alpha-1}\frac{(\ln \xi)^{\beta}}{\xi}|t_2-t_1|+\left(\ln \frac{t_1}{\xi}\right)^{\alpha-1}\frac{(\ln \xi)^{\beta-1}}{\xi}|t_2-t_1|\bigg],
\end{eqnarray*}
where $\xi\in (t_1,t_2)$. Then $I_3\to 0$ as $t_1\to t_2$.\\

Thus we get $\big|(\mathcal{H}x)(t_2)-(\mathcal{H}x)(t_1)\big|\to 0$ as $t_1\to t_2$ which proves that the operator
$\mathcal{H}$ is equicontinuous operator. Hence, by Arzel\`{a}-Ascoli theorem, we conclude that $\mathcal{H}$  is relatively compact on $B_r$. Therefore,  the Schauder fixed point theorem shows that the operator $\mathcal{H}$ has a fixed point which corresponds to the solution of (\ref{main}). This completes the proof.
\end{proof}
Next, we shall show that the following existence and uniqueness result via Banach fixed point theorem.
\begin{theorem}\label{th:uniqueness}
Assume that:
\begin{itemize}
  \item []\textbf{(H2)} Let $f:[1,e]\times\mathbb{R}\to\mathbb{R}$ be a continuous function. There exists a constant $L>0$ such that $|f(t,x)-f(t,y)|\leq L|x-y|$, for each $t\in[1,e]$ and $x,y\in\mathbb{R}$.
  \end{itemize}
Then (\ref{main}) has a unique solution on $[1,e]$,  provided that $L\omega_2<1$.
\end{theorem}
\begin{proof}
As previously proven in \textbf{Step 1} of Theorem \ref{th:existence}, the operator $\mathcal{H}:B_k\to B_k$ defined in (\ref{operator}) is uniformly bounded. Then it remains to show that $\mathcal{H}$ is a contraction mapping.\\

For $x,y\in B_k$, where $B_k=\{x\in C_{\gamma,\ln}([1,e],\mathbb{R}):\|x\|_{\gamma,\ln}\leq k\}$, we have\\

$\left|(\ln t)^{\gamma}((\mathcal{H}x)(t)-(\mathcal{H}y)(t))\right|$
\begin{eqnarray*}
&=&\bigg|(\ln t)^{\gamma}\int_{1}^{t}\left(\ln \frac{t}{s}\right)^{\alpha-1}\mathbb{E}_{\alpha,\alpha}\left(\lambda\left(\ln \frac{t}{s}\right)^{\alpha}\right)
\bigg[\frac{1}{\Gamma(\beta)}\int_{1}^{s}\left(\ln \frac{s}{\tau}\right)^{\beta-1}\big(f(\tau,x(\tau))-f(\tau,y(\tau))\big)\frac{d\tau}{\tau}\\
&-&\frac{(\ln s)^{\beta-1}}{\Gamma(\beta)}\int_{1}^{e}\left(1-\ln\tau\right)^{\beta-1}\big(f(\tau,x(\tau))-f(\tau,y(\tau))\big)\frac{d\tau}{\tau}\bigg]\frac{ds}{s}\bigg|\\
&\leq&\int_{1}^{t}\left(\ln \frac{t}{s}\right)^{\alpha-1}\mathbb{E}_{\alpha,\alpha}\left(\lambda\left(\ln \frac{t}{s}\right)^{\alpha}\right)
\bigg[\frac{1}{\Gamma(\beta)}\int_{1}^{s}\left(\ln \frac{s}{\tau}\right)^{\beta-1}\big|(f(\tau,x(\tau))-f(\tau,y(\tau))\big|\frac{d\tau}{\tau}\\
&+&\frac{(\ln s)^{\beta-1}}{\Gamma(\beta)}\int_{1}^{e}\left(1-\ln\tau\right)^{\beta-1}\big|(f(\tau,x(\tau))-f(\tau,y(\tau))\big|\frac{d\tau}{\tau}\bigg]\frac{ds}{s}\\
&\leq&\int_{1}^{t}\left(\ln \frac{t}{s}\right)^{\alpha-1}\mathbb{E}_{\alpha,\alpha}\left(\lambda\left(\ln \frac{t}{s}\right)^{\alpha}\right)
\bigg[\frac{L}{\Gamma(\beta)}\int_{1}^{s}\left(\ln \frac{s}{\tau}\right)^{\beta-1}|x(\tau)-y(\tau)|\frac{d\tau}{\tau}\\
&+&\frac{L (\ln s)^{\beta-1}}{\Gamma(\beta)}\int_{1}^{e}\left(1-\ln\tau\right)^{\beta-1}|x(\tau)-y(\tau)|\frac{d\tau}{\tau}\bigg]\frac{ds}{s}\\
&\leq&\frac{L}{\Gamma(\beta)}\Bigg\{\int_{1}^{t}\left(\ln \frac{t}{s}\right)^{\alpha-1}\mathbb{E}_{\alpha,\alpha}\left(\lambda\left(\ln \frac{t}{s}\right)^{\alpha}\right)
\bigg[\int_{1}^{s}\left(\ln \frac{s}{\tau}\right)^{\beta-1}(\ln \tau)^{-\gamma}\frac{d\tau}{\tau}\\
&+&(\ln s)^{\beta-1}\int_{1}^{e}\left(1-\ln\tau\right)^{\beta-1}(\ln \tau)^{-\gamma}\frac{d\tau}{\tau}\bigg]\frac{ds}{s}\Bigg\}\|x-y\|_{\gamma,\ln}\\
&\leq&\frac{L\mathbb{B}(\beta,1-\gamma)\mathbb{E}_{\alpha,\alpha}(\lambda)}{\Gamma(\beta)}\left\{\int_{1}^{t}\left(\ln \frac{t}{s}\right)^{\alpha-1}(\ln s)^{\beta-\gamma}\frac{ds}{s}+\int_{1}^{t}\left(\ln \frac{t}{s}\right)^{\alpha-1}(\ln s)^{\beta-1}\frac{ds}{s}\right\}\|x-y\|_{\gamma,\ln}\\
&\leq&\frac{L\mathbb{B}(\beta,1-\gamma)\mathbb{E}_{\alpha,\alpha}(\lambda)}{\Gamma(\beta)}
\bigg(\mathbb{B}(\alpha,1+\beta-\gamma)+\mathbb{B}(\alpha,\beta)\bigg)\|x-y\|_{\gamma,\ln},
\end{eqnarray*}
which implies that $\|\mathcal{H}x-\mathcal{H}y\|_{\gamma,\ln}\leq L\omega_2\|x-y\|_{\gamma,\ln}$. It follows that
$\mathcal{H}$ is a contraction. As a consequence of Banach fixed point theorem, the operator $\mathcal{H}$ has a fixed point which corresponds to the unique solution of (\ref{main}). This completes the proof.
\end{proof}

\section{An example}
Consider the following Langevin equation with two Hadamard fractional derivatives:
\begin{example}
\begin{equation}\label{ex}
\begin{cases}
\prescript{H}{}D_{1,t}^{\frac{3}{4}}(\prescript{H}{}D_{1,t}^{\frac{1}{2}}-1)x(t)=f(t,x(t)),~~t\in[1,e],\\
(\prescript{H}{}D_{1,t}^{\frac{1}{2}}-1) x(e)=0,~~~\prescript{H}{}I_{1^+}^{\frac{1}{2}}x(1^+)=1.
\end{cases}
\end{equation}
Here, $\alpha=\frac{1}{2},\beta=\frac{3}{4},\lambda=1,\gamma=\frac{1}{4}$ and $c_0=1.$ \\

In order to illustrate Theorem \ref{th:existence}, we take $f(t,x)=\sin\frac{1}{81}|x|+\frac{1}{(1+t)^2}$ for all $t\in[1,e]$.\\
Clearly, $|f(t,x)|\leq\frac{1}{81}|x|+\frac{1}{4}$. According to the assumption $(H1)$, we get $L_1=\frac{1}{81}$ and $L_2=\frac{1}{4}$.\\
Thus,
$$L_1\omega_2=\frac{\mathbb{B}\left(\frac{3}{4},\frac{3}{4}\right)}{81\Gamma\left(\frac{3}{4}\right)}\left[\mathbb{B}\left(\frac{1}{2},\frac{3}{2}\right)+\mathbb{B}\left(\frac{1}{2},\frac{3}{4}\right)\right]\mathbb{E}_{\frac{1}{2},\frac{1}{2}}(1)=0.06772116862~\mathbb{E}_{\frac{1}{2},\frac{1}{2}}(1).$$
Using Lemma \ref{Gorenfloetal}, we get
$$\mathbb{E}_{\frac{1}{2},\frac{1}{2}}(1)=\frac{1}{\Gamma\left(\frac{1}{2}\right)}+\mathbb{E}_{\frac{1}{2},1}(1)=\frac{1}{\sqrt{\pi}}+\textnormal{erfc(-1)}\times e=5.573170227.$$ Therefore, $L_1\omega_2=0.3774216007<1$, and according to  Theorem \ref{th:existence}, we conclude that

the Langevin equation (\ref{ex}) with $f(t,x)=\sin\frac{1}{81}|x|+\frac{1}{(1+t)^2}$ has at least one solution on $[1,e]$.\\

For the illustration of Theorem \ref{th:uniqueness}, let us take $f(t,x)=\frac{|x|}{(99+t^2)(1+|x|)}$ for all $t\in[1,e]$.
Obviously, $|f(t,x)-f(t,y)|\leq\frac{1}{100}|x-y|$. Thus,  the assumption $(H2)$ implies that  $L=\frac{1}{100}$.\\

The direct computations give
$$L\omega_2=\frac{\mathbb{B}\left(\frac{3}{4},\frac{3}{4}\right)}{100\Gamma\left(\frac{3}{4}\right)}\left[\mathbb{B}\left(\frac{1}{2},\frac{3}{2}\right)+\mathbb{B}\left(\frac{1}{2},\frac{3}{4}\right)\right]\mathbb{E}_{\frac{1}{2},\frac{1}{2}}(1)=0.05485414658~ \mathbb{E}_{\frac{1}{2},\frac{1}{2}}(1).$$
Hence $L\omega_2=0.3057114966<1.$ According to Theorem \ref{th:uniqueness}, the Langevin equation (\ref{ex}) with $f(t,x)=\frac{|x|}{(99+t^2)(1+|x|)}$ has a unique solution on $[1,e]$.
\end{example}

\end{document}